\documentclass{amsart}

\title{Witt vectors and truncation posets}
\author{Vigleik Angeltveit}
\email{vigleik.angeltveit@anu.edu.au}
\address{Mathematical Sciences Institute \\
Australian National University \\
Acton, ACT 2601 \\
Australia}

\usepackage{amsxtra}
\usepackage{amsfonts}
\usepackage[latin1]{inputenc}
\usepackage{graphicx}
\usepackage{amsmath,amssymb,latexsym,amsthm,mathrsfs}
\usepackage[all]{xy}
\usepackage{pstricks}
\usepackage{verbatim}

\newtheorem{theorem}{Theorem}[section]

\newtheorem{lemma}[theorem]{Lemma}

\newtheorem{proposition}[theorem]{Proposition}

\theoremstyle{definition}
\newtheorem{definition}[theorem]{Definition}

\newtheorem{remark}[theorem]{Remark}
\newtheorem{example}[theorem]{Example}

             \newcommand{\bN}{\mathbb{N}}  
      \newcommand{\bW}{\mathbb{W}}   \newcommand{\bZ}{\mathbb{Z}}

\newcommand{\xto}{\xrightarrow}
\newcommand{\xfrom}{\xleftarrow}
\newcommand{\wh}{\widehat}

\newcommand{\op}{{\mathrm{op}}}
\newcommand{\join}{\textnormal{join}}
\newcommand{\TP}{\mathcal{TP}}

\begin{document}

\begin{abstract}
One way to define Witt vectors starts with a truncation set $S \subset \bN$. We generalize Witt vectors to truncation posets, and show how three types of maps of truncation posets can be used to encode the following six structure maps on Witt vectors: addition, multiplication, restriction, Frobenius, Verschiebung and norm.
\end{abstract}

\maketitle

\section{Introduction}
Classically the theory of Witt vectors comes in two flavors, the $p$-typical Witt vectors $W(k;p)$ and the big Witt vectors $\bW(k)$. Those are special cases of Witt vectors defined using a truncation set $S$, and the extra flexibility coming from varying the truncation set has proven quite useful.

In this paper we take the use of truncation sets one step further by introducing \emph{truncation posets}, and redevelop the foundations of Witt vectors from this point of view. The existence of all the usual structure maps of Witt vectors is easy to establish using this formalism. We give explicit formulas on ghost coordinates and isolate all the necessary congruences in a single lemma due to Dwork (Lemma \ref{l:Dwork}).

Recall that a \emph{truncation set} is a set $S \subset \bN=\{1,2,\ldots\}$ which is closed under division. Given a truncation set $S$ and a commutative ring $k$, one can define the ring $\bW_S(k)$ of Witt vectors. As a set this is $k^S$, and the addition and multiplication maps are determined by requiring that the ghost map $w : \bW_S(k) \to k^S$ is a ring map, functorially in the ring $k$. With $S=\{1,p,p^2,\ldots\}$ this recovers the $p$-typical Witt vectors and with $S=\bN$ it recovers the big Witt vectors.

In recent work related to various algebraic $K$-theory calculations more general truncation sets have come up. For example, when studying the algebraic $K$-theory of the ring $k[x_1,\ldots,x_n]/(x_1^{a_1},\ldots,x_n^{a_n})$ in \cite{AGHL14} it turned out to be natural to consider certain subsets of $\bN^n$, and in \cite{An15} where we calculate the algebraic $K$-theory of $k \langle x_1,\ldots,x_n \rangle/m^a$, the polynomial ring in $n$ non-commuting variables modulo the $a$'th power of the ideal $m=(x_1,\ldots,x_n)$, we were led to consider certain subsets of the set of words in $n$ letters.

It is always possible to unpack truncation poset Witt vectors as a product of classical Witt vectors in a unique way (up to permutations), and to describe the maps in terms of the structure maps of classical Witt vectors. But this unpacking is messy, and naturally defined maps of truncation poset Witt vectors have to be divided into cases when considering only the classical Witt vectors. We claim that by considering truncation posets the above-mentioned $K$-theory calculations become somewhat easier to carry out, and the results become significantly easier to state.

For example, when considering the truncation posets in $\bN^n$ from \cite{AGHL14}, the Verschiebung $V^i_r : \bW_{S/(1,\ldots,r,\ldots,1)}(k) \to \bW_S(k)$ splits as a product of Verschiebung maps on each classical Witt vector factor, but which $V_j$ is used on which factor is somewhat complicated (compare \cite[Prop.\ 2.7]{AGHL14}). We will return to another example in \cite{An15}, where we calculate $K_*(A,m)$ for $A=k\langle x_1,\ldots,x_n \rangle/m^a$ when $k$ is a perfect field of positive characteristic. We find that the $K$-theory in degree $2q-1$ and $2q$ sit in an exact sequence
\[
 0 \to K_{2q}(A,m) \to \bW_{S_n(a,aq)}(k) \xto{V^1_a} \bW_{S_n(1,aq)}(k) \to K_{2q-1}(A,m) \to 0
\]
for certain truncation posets $S_n(a,aq)$ and $S_n(1,aq)$ (see also Example \ref{ex:words2} and \ref{ex:wordmap} in this paper). Compare 
this to the description in the case $a=2$ found by Lindenstrauss and McCarthy \cite[Theorem 7.3 andf 7.4]{LiMc}.

Given a truncation poset $S$ as in Definition \ref{d:gentrun} below and a commutative ring $k$, we will define the $S$-Witt vectors $\bW_S(k)$ to be $k^S$ as a set. We will then make the collection of truncation posets into a category, and $S \mapsto \bW_S(k)$ into a functor, in a number of different ways by considering three types of maps.

The first type of map, which we call an $R$-map, is most general. Given an $R$-map $f : S \to T$ of truncation posets we get an induced map $f^* : \bW_T(k) \to \bW_S(k)$. By varying $S$, $T$ and $f$ this recovers all composites of the classical restriction and Frobenius maps, as well as diagonal maps. Classically the restriction and Frobenius maps are defined in rather different ways, so it is perhaps surprising that the two definitions can be unified in this way.

The second type of map, which we call a $T$-map, is an $R$-map satisfying certain additional conditions. Given a $T$-map $f : S \to T$ of truncation posets we get an induced map $f_\oplus : \bW_S(k) \to \bW_T(k)$, and by varying $S$, $T$ and $f$ this recovers all composites of the addition map and Verschiebung maps on the classical Witt vectors.

We can combine these two kinds of maps and define a category $\TP^{TR}$. An object of $\TP^{TR}$ is a truncation poset, and a morphism is an equivalence class of spans
\[
 S \xfrom{f} A \xto{g} T
\]
where $f$ is an $R$-map and $g$ is a $T$-map. This is similar to the definition of a $G$-Mackey functor in terms of spans of finite $G$-sets for a finite group $G$. See Theorem \ref{t:TRfunctor} in the body of the paper.

Finally, we define a third type of map of truncation posets that we call an $N$-map. This is an $R$-map satisfying certain (stronger) additional conditions. Given an $N$-map $f : S \to T$ we get an induced map $f_\otimes : \bW_S(k) \to \bW_T(k)$ which encodes all composites of the multiplication map and norm maps on the classical Witt vectors.

We can combine all three kinds of maps to define a category $\TP^{TNR}$ of truncation posets with transfer, norm and restriction. We then have the following result, which we also restate as Theorem \ref{t:mainbody}.

\begin{theorem} \label{t:main}
Let $k$ be a commutative ring. There is a functor
\[
 \bW(k) : \TP^{TNR} \to Set
\]
given on objects by $S \mapsto \bW_S(k)$ which encodes all addition, multiplication, restriction, Frobenius, Verschiebung and norm maps of ordinary Witt vectors.
\end{theorem}

While one can argue that some category encoding all of this information must exist for formal reasons, our category $\TP^{TNR}$ has a very concrete description in terms of generators and relations, and it is easy to perform calculations on ghost coordinates.

We make some remarks.

\begin{remark}
The norm map is perhaps less classical than the other maps encoded by $\TP^{TNR}$. It can be thought of as a multiplicative version of the Verschiebung. Its existence can be deduced from Brun's paper \cite{Br05}, but see \cite{An_norm} for a concrete definition with explicit formulas.
\end{remark}

\begin{remark} \label{r:Tambarabispans}
The machinery developed in this paper is similar in flavor to that of \emph{Tambara functors} (see \cite{Ta93} or \cite{St}). In fact, Tambara called what has become known as a Tambara functor a $TNR$-functor. But there are some differences.

First, there is no analogue of the restriction map in the context of equivariant stable homotopy theory unless one is willing to consider \emph{cyclotomic spectra}. What topologists usually refer to as a restriction map corresponds to the Frobenius map of Witt vectors. To avoid confusion we will avoid the conflicting terminology from algebraic topology in this paper, although we do borrow the acronym $TNR$.

And second, a Tambara functor can be defined as a functor from the category of \emph{bispans}
\[
 X \leftarrow A \to B \to Y,
\]
of finite $G$-sets, and while the definition of composition of two bispans is somewhat complicated it is possible to represent any composite of restrictions (which we should call Frobenius), norms and transfers, in any order, as a bispan. In our case $\TP^{TNR}$ is also built from three types of maps, but it is not true that any map in $\TP^{TNR}$ can be represented by a bispan.

One might argue that this indicates that our definition of a truncation poset is too general. We remedy this by defining a subcategory $\TP^{TNR}_\join$ containing only certain especially nice truncation posets, and show that any map in $\TP^{TNR}_\join$ can indeed be represented by a bispan. But note that the truncation posets that show up in $K$-theory calculations are not usually in $\TP^{TNR}_\join$.
\end{remark}

\subsection{Outline}
We start in Section \ref{s:gentrun} by defining the main new player, the truncation poset. In Section \ref{s:genWitt} we describe how to generalize Witt vectors from ordinary truncation sets to truncation posets, and explain how $S \mapsto \bW_S(k)$ defines a functor out of each of the three categories $\TP^T$, $\TP^N$ and $(\TP^R)^\op$.

In Section \ref{s:Mackey} we combine the category $\TP^T$ with $(\TP^R)^\op$ by considering the category freely generated by maps in $\TP^T$ and $(\TP^R)^\op$, modulo certain explicit relations. Any map in $\TP^{TR}$ can be described by a \emph{span} of truncation posets where the first leg is in $\TP^R$ and the second leg is in $\TP^T$. Then $S \mapsto \bW_S(k)$ becomes a functor from $\TP^{TR}$ to sets. This is similar in flavor to the definition of a Mackey functor.

In Section \ref{s:Tambara}, which is significantly more difficult both because of the difficulty with commuting an $R$-map past an $N$-map and because of the combinatorics involved in defining an exponential diagram, we combine all three of the categories $\TP^T$, $\TP^N$ and $(\TP^R)^\op$ and show that $S \mapsto \bW_S(k)$ is a functor from $\TP^{TNR}$ to sets. This is similar in flavor to the definition of a Tambara functor, but see Remark \ref{r:Tambarabispans} above.

Finally, in Section \ref{s:bispans} we show that if we restrict our attention to certain especially nice truncation posets we can define a category $\TP^{TNR}_\join$ where every morphism can in fact be represented, in an essentially unique way, by a bispan of truncation posets. We finish by comparing functors out of a particular subcategory of $\TP^{TNR}_\join$ to Tambara functors for a finite cyclic group.

\subsection{Acknowledgements}
The definition of a truncation poset was inspired by the author's joint work with Gerhardt, Hill and Lindenstrauss \cite{AGHL14}, and by our subsequent attempt at finding a common generalization of the $n$-dimensional Witt vectors we introduced in \cite{AGHL14} and the Dress-Siebeneicher Witt vectors from \cite{DrSi88}. One approach we outlined contained a definition that was quite similar to the definition of a truncation poset presented here. 

This paper was also inspired by the author's joint work with Anna Marie Bohmann on graded Tambara functors \cite{AnBo}, and by conversations with Ayelet Lindenstrauss and Lars Hesselholt about algebraic $K$-theory calculations. We have also borrowed some of the Witt vector formalism from Hesselholt's article \cite{He15}. The author would also like to thank James Borger and Arnab Saha for interesting conversations about Witt vectors, and Chuck Weibel for suggesting the name truncation poset.

This work was supported by the Australian Research Council grant DP120101399.

\section{Truncation posets and maps between them} \label{s:gentrun}
If $S$ is a partially ordered set we write $s \mid t$ rather than $s \leq t$ for the partial order. We will consider $\bN$ as a partially ordered set ordered by division.

\subsection{Ordinary truncation sets and classical Witt vectors}
We will refer to anything defined in terms of ordinary truncation sets as ``classical''. With the exception of the norm map, this material is well known. See for example \cite[Section 1]{He15}. Recall that $S \subset \bN$ is a truncation set if $s \in S$ and $t \mid s$ implies $t \in S$. For a commutative ring $k$, the ring of $S$-Witt vectors $\bW_S(k)$ is defined to be $k^S$ as a set. The addition and multiplication maps are defined by the requirement that the ghost map
\[
 w : \bW_S(k) \to k^S
\]
defined by
\[
 (a_s) \mapsto \langle x_s \rangle \qquad x_s = \sum_{d \mid s} da_d^{s/d}
\]
is a ring map, functorially in the ring $k$. We make the standing assumption that everything in this paper is required to be functorial in $k$.

We will need the following constructions. If $n \in \bN$ and $S$ is a truncation set, let
\[
 S/n = \{t \in \bN \quad | \quad nt \in S\}.
\]
This is another truncation set. The classical Frobenius map is defined as the map
\[
 F_n : \bW_S(k) \to \bW_{S/n}(k)
\]
which is given on ghost coordinates by $\langle x_s \rangle \mapsto \langle y_t \rangle$ with $y_t = x_{nt}$.

There is also a map going the other way. The classical Verschiebung map is defined as the map
\[
 V_n : \bW_{S/n}(k) \to \bW_S(k).
\]
which is given on Witt coordinates by $(b_t) \mapsto (a_s)$ with $a_s = b_{s/n}$ if $n \mid s$ and $0$ if $n \nmid s$. Alternatively it can be defined on ghost coordinates by $\langle y_t \rangle \mapsto \langle x_s \rangle$ with $x_s = ny_{s/n}$ if $n \mid s$ and $0$ if $n \nmid s$.

For $n \in \bN$ we let $\langle n \rangle$ denote the truncation set of divisors of $n$. Given a truncation set $S$ we get another truncation set
\[
 \langle n \rangle S = \{t \in \bN \quad | \quad t=es \textnormal{ for some $e \mid n$, $s \in S$}\}.
\]
It follows immediately that $(\langle n \rangle S)/n = S$, so in particular we have a Verschiebung map $V_n : \bW_S(k) \to \bW_{\langle n \rangle S}(k)$. But by \cite{An_norm} we also have a norm map (the ``classical'' norm)
\[
 N_n : \bW_S(k) \to \bW_{\langle n \rangle S}(k).
\]
This can be defined on ghost coordinates by $\langle x_s \rangle \mapsto \langle y_t \rangle$ where
\[
 y_t = x_{t/g}^g, \quad g=\gcd(n,t).
\]

Finally, if $T \subset S$ is another truncation set there is a classical restriction map
\[
 R^S_T : \bW_S(k) \to \bW_T(k).
\]
This can be defined either on Witt coordinates by $(R^S_T(a_s))_t = a_t$ or on ghost coordinates by $(R^S_T \langle x_s \rangle)_t = x_t$.

Note that specifying $n \in \bN$ and the source $\bW_S(k)$ uniquely determines the target of the norm map $N_n$ but not the target of the Verschiebung map $V_n$. However, given a truncation set $T$ with $T/n=S$ it is true that $\langle n \rangle S \subset T$, and that the diagram
\[ \xymatrix{
 \bW_S(k) \ar[r]^-{V_n} \ar[rd]_-{V_n} & \bW_T(k) \ar[d]^{R^T_{\langle n \rangle S}} \\
 & \bW_{\langle n \rangle S}(k)
} \]
commutes. Given $(a_s) \in \bW_S(k)$, the image $V_n(a_s) \in \bW_T(k)$ is padded with zeroes. But because the formula for $N_n(a_n)$ is more complicated, padding with zeroes does not work in this case. See Example \ref{ex:notcommutingRpastN} below for a concrete example of this problem with the norm map.

\subsection{Truncation posets}
Inspired by \cite{AGHL14} and subsequent discussions with Gerhardt, Hill and Lindenstrauss we make the following definition.

\begin{definition} \label{d:gentrun}
A \emph{truncation poset} is a partially ordered set $S$ together with a function
\[
 | - | : S \to \bN
\]
satisfying the following properties.
\begin{enumerate}
 \item If $s \mid t$ then $|s| \mid |t|$.
 \item If $d \mid |s|$ then there is a unique $t \in S$ with $t \mid s$ and $|t| = |s|/d$. In particular there is a unique $t \in S$ with $t \mid s$ and $|t|=1$.
 \item If $s \in S$ and $d \in \bN$ there is at most one $t \in S$ with $s \mid t$ and $|t|=d|s|$.
\end{enumerate}
\end{definition}

In fact all truncation posets split up as a disjoint union.

\begin{lemma} \label{l:splittingofS}
Let $S$ be a truncation poset. Then there is a unique splitting $S = \coprod S_i$ with each $S_i$ isomorphic to an ordinary truncation set via $|-|$.
\end{lemma}

\begin{proof}
Let $I = \{ i \in S \quad | \quad |i|=1\}$. Then it is clear from the definition that $S = \coprod\limits_{i \in I} S_i$ with $S_i = \{s \in S \quad | \quad i \mid s\}$.
\end{proof}

We call each $S_i \subset S$ as in the above lemma a \emph{connected component} of $S$. For ease of notation we will sometimes write $t/s$ for the natural number $\frac{|t|}{|s|}$ and $s/d$ for the unique $t \in S$ with $t \mid s$ and $|t|=|s|/d$. If there is a possibility for confusion we will write $|s|_S$ for $|s|$.

\begin{remark}
Suppose we are given a poset $S$ and a natural number $t/s$ for each $s \mid t$ in $S$. Then there is at most one way to define $|-|$ in such a way that $S$ becomes a truncation poset. Indeed, we must have
\[
 |s|=\max\{s/t \quad | \quad t \mid s\}.
\]
It is not hard to translate the conditions in Definition \ref{d:gentrun} into conditions on the natural numbers $t/s$ for each $s \mid t$ in a poset $S$.
\end{remark}

\begin{definition}
Let $S$ be a truncation poset and let $k$ be a commutative ring. The $S$-Witt vectors of $k$, denoted $\bW_S(k)$, is the set $k^S$. The ghost map is the map
\[
 w : \bW_S(k) \to k^S
\]
sending the vector $(a_s)$ to the vector $\langle x_s \rangle$ with
\[
 x_s = \sum_{t \mid s} |t| a_t^{s/t}.
\]
\end{definition}

\begin{lemma}
If $S$ and $T$ are truncation posets then so is $S \coprod T$, and there is a canonical isomorphism
\[
 \bW_{S \coprod T}(k) \cong \bW_S(k) \times \bW_T(k).
\]
Moreover, this isomorphism is compatible with the canonical isomorphism $k^{S \coprod T} \cong k^S \times k^T$ under the ghost map.
\end{lemma}

\begin{proof}
This is clear by inspection of the definitions.
\end{proof}

We will describe the various structure maps that exist in Section \ref{s:genWitt} below. But first we present a series of examples of truncation posets and define the maps of truncation posets we will need.

\begin{example}
An ordinary truncation set $S \subset \bN$ is a truncation poset with $|-|$ defined to be the identity map.
\end{example}

\begin{example} \label{ex:nN}
Let $n \in \bN$. The subset $n \bN \subset \bN$ is a truncation poset with $|s|_{n\bN}=s/n$. The multiplication by $n$ map $\bN \to n \bN$ is an isomorphism of truncation posets.
\end{example}

The next two examples appeared in \cite{AGHL14}, the first one explicitly and the second one implicitly.

\begin{example} \label{ex:N_to_the_n}
We consider $\bN^n$ as a partially ordered set where $(t_1,\ldots,t_n) \mid (s_1,\ldots,s_n)$ if there is some $d \in \bN$ with $(dt_1,\ldots,dt_n) = (s_1,\ldots,s_n)$. A subset $S \subset \bN^n$ which is closed under division in the sense that if $(s_1,\ldots,s_n) \in S$ and $d \mid s_i$ for all $1 \leq i \leq n$ then $(s_1/d,\ldots,s_n/d) \in S$, is a truncation poset with $|(s_1,\ldots,s_n)| = \gcd(s_1,\ldots,s_n)$.
\end{example}

\begin{example} \label{ex:N_to_the_n2}
Fix positive integers $a_1,\ldots,a_n$. A subset $S \subset \bN^n$ which satisfies $a_i \mid s_i$ for all $(s_1,\ldots,s_n) \in S$ and $1 \leq i \leq n$, and which is closed under division by $n$-tuples satisfying the same condition, is a truncation poset with $|s_1,\ldots,s_n|=\gcd(\frac{s_1}{a_1},\ldots,\frac{s_n}{a_n})$.
\end{example}

\begin{example} \label{ex:words}
Put a partial order on the set of words in $n$ letters by saying $w_1 \mid w_2$ if $w_2 = w_1^d$ for some $d \in \bN$. Define $|w_2|=d$ if $w_2 = w_1^d$ with $w_1$ irreducible (meaning $w_1$ is not a power of a shorter word). For example, $|x_1 x_2 x_1 x_2|=2$. Then a set $S$ of words which is closed under division is a truncation poset.
\end{example}

The next example is central to the calculation of the algebraic $K$-theory of a truncated polynomial ring in non-commuting variables, see \cite{An15}.

\begin{example} \label{ex:words2}
Fix a positive integer $a$ and consider words in $n$ letters of length divisible by $a$, modulo the equivalence relation given by cyclically permuting blocks of $a$ letters. A set of such words which is closed under division in the same sense as in the previous example, and with $|-|$ defined in the same way, is a truncation poset. For example, if $a=1$ the words $x_1x_2x_1x_2$ and $x_2x_1x_2x_1$ are equivalent but if $a=2$ or $a=4$ they are not. If $a=1$ or $a=2$ then $|x_1x_2x_1x_2|=2$ but if $a=4$ then $|x_1x_2x_1x_2|=1$. 
\end{example}

We could go on, but we hope the above examples have convinced the reader that truncation posets are in rich supply.

\subsection{Maps of truncation posets}
The structure maps for Witt vectors will come from maps of truncation posets. We make the following definition.

\begin{definition}
A \emph{map} $f : S \to T$ of truncation posets is a map of sets such that if $s_1 \mid s_2$ then $f(s_1) \mid f(s_2)$ and $\frac{|f(s_2)|}{|f(s_1)|} = \frac{|s_2|}{|s_1|}$.
\end{definition}

If $S = \coprod\limits_{i \in I} S_i$ is a splitting into ordinary truncation sets as in Lemma \ref{l:splittingofS} above with $i \in S_i$ and $|i|=1$, then $f : S \to T$ is uniquely determined by $f(i)$ for $i \in I$.

For maximal similarity with the later definitions we will sometimes call a map of truncation posets an $R$-map. It is clear that such maps compose and that we get a category $\TP^R$ of truncation posets and $R$-maps.

\begin{example}
Let $f : T \subset S$ be an inclusion of ordinary truncation sets. Then $f$ is a map of truncation posets.
\end{example}

\begin{example} \label{ex:multn}
Let $S$ be an ordinary truncation set and let $T = n\bN \cap S \subset S$. Then $T$ is a truncation poset as in Example \ref{ex:nN} above and the inclusion $f : T \to S$ is a map of truncation posets. Moreover, the map
\[
 \frac{1}{n} : T \to S/n
\]
is an isomorphism of truncation posets. We will switch back and forth between thinking about the inclusion $n\bN \cap S \subset S$ and the multiplication by $n$ map $S/n \to S$.
\end{example}

\begin{example} \label{ex:multn2}
As a special case of the previous example, let $S$ be an ordinary truncation set. Then the multiplication by $n$ map $S \xto{n} \langle n \rangle S$ is a map of truncation posets.
\end{example}

\begin{example}
Let $S$ be any truncation poset. Then the fold map $\nabla : S \coprod S \to S$ is a map of truncation posets.
\end{example}

Any map of truncation posets will induce a map between Witt vectors, and some maps will induce two or three different maps. For the extra maps we need additional conditions.

\begin{definition} \label{d:Tmap}
A map $f : S \to T$ of truncation posets is a \emph{$T$-map} ($T$ for \emph{transfer}, not for the target of the map) if it satisfies the following additional conditions.
\begin{enumerate}
 \item For every $s \in S$ and $t' \in T$ with $f(s) \mid t'$ there exists an $s' \in S$ with $s \mid s'$ and $f(s')=t'$. \label{cond:fib}
 \item For every $t \in T$ the set $f^{-1}(t)$ is finite.
\end{enumerate}

If $f$ satisfies the first condition (but not necessarily the second) we say that $f$ is a \emph{fibration}.
\end{definition}

It is clear that we get a category $\TP^T$ of truncation posets and $T$-maps. To get a better understanding of $\TP^T$, we first consider ordinary truncation sets.

\begin{lemma} \label{l:Tmapclassical}
Suppose $U$ and $V$ are ordinary truncation sets. Then $g : U \to V$ is a $T$-map if and only if $U=V/n$ and $g$ is multiplication by $n$ as in Example \ref{ex:multn}.
\end{lemma}

\begin{proof}
Let $n=g(1)$, and note that $g$ is injective. We can define a map $g' : U \to V/n$ by $g'(u) = \frac{g(u)}{n}$, and then $g$ factors as $U \xto{g'} V/n \xto{n} V$. The map $g'$ is injective and satisfies $g'(1)=1$. Now condition (\ref{cond:fib}) in Definition \ref{d:Tmap} implies that $g'$ is also surjective.
\end{proof}

\begin{lemma} \label{l:decomposeTmap}
Suppose $f : S \to T$ is a $T$-map of truncation posets. Decompose $S$ and $T$ as $S = \coprod_i S_i$ and $T = \coprod_j T_j$ with each $S_i$ and $T_j$ isomorphic to an ordinary truncation set as in Lemma \ref{l:splittingofS}. Then $f$ is the coproduct of maps
\[
 S_i \xto{f_i} T_j \subset T,
\]
and each $f_i$ is isomorphic to a $T$-map of classical truncation sets of the form $n_i : V_j/n_i \to V_j$.
\end{lemma}

\begin{proof}
This follows immediately from Lemma \ref{l:Tmapclassical}.
\end{proof}

We need one more version of a map of truncation posets.

\begin{definition} \label{d:Nmap}
A map $f : S \to T$ of truncation posets is an \emph{$N$-map} ($N$ for \emph{norm}) if it satisfies the following additional conditions.
\begin{enumerate}
\item For every $s \in S$ and $t' \in T$ with $t'$ in the same connected component as $f(s)$ there exists an $s' \in S$ in the same connected component as $s$ with $t' \mid f(s')$. \label{cond:strongfib}
\item For every $t \in T$ the set
\[
 \widehat{f^{-1}}(t) = \textnormal{minimal elements of } \{s \in S \quad | \quad t \mid f(s)\}
\]
is finite.
\end{enumerate}
\end{definition}

Here an $s \in S$ with $t \mid f(s)$ is \emph{minimal} if there is no $s' \in S$ with $s' \mid s$, $s' \neq s$, and $t \mid f(s')$. The set $\widehat{f^{-1}}(t)$ contains exactly one element in each connected component of $S$ that maps to the connected component of $T$ containing $t$.

We note that being an $N$-map is a stronger condition than being a $T$-map. Indeed, if $f : S \to T$ is an $N$-map and $f(s) \mid t'$ we get an $s''$ in the same connected component as $s$ with $t' \mid f(s'')$. But then $f(s) \mid t' \mid f(s'')$ implies that $s \mid s''$ and that $d = \frac{|f(s'')|}{|t'|} \mid |s''|$. Hence we can let $s' = \frac{s''}{d}$. It follows that $s \mid s'$ and that $f(s')=t'$.

The generalization $\widehat{f^{-1}}(t)$ of inverse image is compatible with composition in the following sense:

\begin{lemma}
Suppose we have a composite of $N$-maps $S \xto{f} T \xto{g} U$. Then the composite
\[
 \textnormal{subsets of $U$} \xto{\widehat{g^{-1}}} \textnormal{subsets of $T$} \xto{\widehat{f^{-1}}} \textnormal{subsets of $S$}
\]
agrees with
\[
 \textnormal{subsets of $U$} \xto{\widehat{(g \circ f)^{-1}}} \textnormal{subsets of $S$}.
\]
\end{lemma}

\begin{proof}
This is a straightforward verification. To show that $\widehat{(g \circ f)^{-1}}(u)$ is contained in $\widehat{f^{-1}}(\widehat{g^{-1}}(u))$, take $s \in \widehat{(g \circ f)^{-1}}(u)$. Then $\gcd(|s|, \frac{|g(f(s))|}{|u|})=1$. Now let $t = \frac{f(s)}{d}$, where $d = \gcd(|f(s)|, \frac{|g(f(s))|}{|u|})$. It follows that $t \in \widehat{g^{-1}}(u)$ and that $s \in \widehat{f^{-1}}(t)$, so $s \in \widehat{f^{-1}}(\widehat{g^{-1}}(u))$. The opposite inclusion is similar.
\end{proof}

We get a category $\TP^N$ of truncation posets and $N$-maps. We note the following consequence of the definition of an $N$-map.

\begin{lemma} \label{l:minimaldiv}
Suppose $f : S \to T$ is an $N$-map. Given $t, t' \in T$ with $t \mid t'$ and $s \in \widehat{f^{-1}}(t)$ there is a unique $s' \in \widehat{f^{-1}}(t')$ with $s \mid s'$. Conversely, given $s' \in \widehat{f^{-1}}(t')$ there is a unique $s \in \widehat{f^{-1}}(t)$ with $s \mid s'$. Moreover, $\frac{|t|}{|s|}$ divides $\frac{|t'|}{|s'|}$.
\end{lemma}

\begin{proof}
Given $t' \in T$ with $t \mid t'$, the definition of an $N$-map gives us some $s'$ in the same connected component as $s$ with $t' \mid f(s')$. If we require $s'$ to be minimal then $s'$ is unique.

But then $t \mid f(s')$ and we get an element $s'' \in \widehat{f^{-1}}(t)$ defined by $s'' = \frac{s'}{d}$ with $d = \gcd(|s'|, \frac{|f(s')|}{|t|})$. It follows that $s=s''$ since they are in the same connected component and satisfy the same minimality condition. Hence $s=s'' \mid s'$.

For the converse, note that $s' \in \widehat{f^{-1}}(t')$ also satisfies $t \mid f(s')$. We can then define $s$ by requiring that $s \mid s'$ and $t \mid f(s)$, and that $s$ is minimal.
\end{proof}

To get a better understanding of $\TP^N$, we first consider ordinary truncation sets.

\begin{lemma} \label{l:Nmapclassical}
Suppose $U$ and $V$ are ordinary truncation sets. Then $g : U \to V$ is an $N$-map if and only if $V=\langle n \rangle U$ and $g$ is multiplication by $n$ as in Example \ref{ex:multn2}.
\end{lemma}

\begin{proof}
Let $n=g(1)$. Then $g$ factors as $U \xto{n} \langle n \rangle U \xto{g'} V$. Now $g'$ is automatically injective, and condition (\ref{cond:strongfib}) in Definition \ref{d:Nmap} implies that $g'$ is also surjective.
\end{proof}

For an $N$-map of truncation posets we then get the following:

\begin{lemma} \label{l:decomposeNmap}
Suppose $f : S \to T$ is an $N$-map. Decompose $S$ and $T$ as $S = \coprod_i S_i$ and $T=\coprod_j T_j$ with each $S_i$ and $T_j$ isomorphic to an ordinary truncation set as in Lemma \ref{l:splittingofS}. Then $f$ is the coproduct of maps
\[
 S_i \xto{f_i} T_j \subset T,
\]
and each $f_i$ is isomorphic to a map of the form $U_i \xto{n_i} \langle n_i \rangle U_i$ as in Example \ref{ex:multn2}.

Moreover, only finitely many $S_i$ map to each $T_j$.
\end{lemma}

\begin{example}
Fix positive integers $a_1,\ldots,a_n$ and $b_1,\ldots,b_n$ with each $a_i \mid b_i$. Also fix $N \in \bN \cup \{\infty\}$. Let $S \subset \bN^n$ be the following truncation poset:
\[
 S = \{(s_1,\ldots,s_n) \in \bN^n \quad | \quad a_i \mid s_i \textnormal{ and } s_1+\ldots+s_n \leq N \}
\]
with
\[
 |(s_1,\ldots,s_n)| = \gcd \Big( \frac{s_1}{a_1},\ldots,\frac{s_n}{a_n} \Big)
\]
as in Example \ref{ex:N_to_the_n2}.

Similarly, let $T \subset \bN^n$ be the truncation poset
\[
 T = \{(t_1,\ldots,t_n) \in \bN^n \quad | \quad b_i \mid t_i \textnormal{ and } t_1+\ldots+t_n \leq N \}
\]
with
\[
 |(t_1,\ldots,t_n)| = \gcd \Big( \frac{t_1}{b_1},\ldots,\frac{t_n}{b_n} \Big).
\]
Then the inclusion $T \subset S$ is a $T$-map but not generally an $N$-map.
\end{example}

\begin{example} \label{ex:wordmap}
Fix positive integers $a$ and $b$ with $a \mid b$. Also fix $N \in \bN \cup \{\infty\}$. Let $S$ be the following truncation poset:
\[
 S = \{ \textnormal{words in $x_1,\ldots,x_n$ of length $\leq N$ and word length divisible by $a$} \}/\sim_a
\]
where $\sim_a$ is the equivalence relation given by cyclically permuting blocks of $a$ letters as in Example \ref{ex:words2}.

Similarly, let $T$ be the truncation poset
\[
 T = \{ \textnormal{words in $x_1,\ldots,x_n$ of length $\leq N$ and word length divisible by $b$} \}/\sim_b.
\]
Then there is a natural map $T \to S$ sending the equivalence class $[w]_{\sim_b}$ of a word to $[w]_{\sim_a}$, and this map is a $T$-map but not generally an $N$-map. 
\end{example}

\begin{remark}
Note that we never require $|f(s)|=|s|$, and that in the interesting examples $|s|$ is not defined in the most obvious way.
\end{remark}

\section{Witt vectors as functors from $\TP$} \label{s:genWitt}
Before we say anything about how to combine the categories $\TP^T$, $\TP^N$ and $(\TP^R)^\op$ we describe the Witt vectors as a functor from each individual category. We will use the following result, which Hesselholt \cite{He15} attributes to Dwork in the case of an ordinary truncation set.

\begin{lemma} \label{l:Dwork}
Let $S$ be a truncation poset and let $k$ be a commutative ring. Suppose that for every prime $p$, there exists a ring homomorphism $\phi_p : k \to k$ such that $\phi_p(a) \equiv a^p \mod p$. Then $\langle x_s \rangle$ is in the image of the ghost map if and only if $x_s \equiv \phi_p(x_{s/p}) \mod p^{\nu_p(|s|)}$ for every $p$ and every $s \in S$ with $p \mid |s|$.
\end{lemma}

This is a direct consequence of the ordinary Dwork Lemma and the fact that Witt vectors for truncation posets decompose as products of classical Witt vectors.

\subsection{Restriction and Frobenius maps}
The maps in $\TP^R$ will do triple duty, encoding restriction and Frobenius maps as well as diagonal maps. We start with the following definition, which we justify below.

\begin{definition} \label{d:RFdiagmap}
Let $f : S \to T$ be a map of truncation posets. Then
\[
 f^* : \bW_T(k) \to \bW_S(k)
\]
is defined to be the unique map of sets, natural in $k$, such that the diagram
\[ \xymatrix{
 \bW_T(k) \ar[r]^-w \ar[d]_{f^*} & k^T \ar[d]^{f^*_w} \\
 \bW_S(k) \ar[r]^-w & k^S
} \]
commutes. Here $f^*_w$ is defined by $(f^*_w \langle x_t \rangle)_s = x_{f(s)}$.
\end{definition}

It is clear that if this defines a map on Witt vectors then $(g \circ f)^* = f^* \circ g^*$ because this holds on ghost coordinates, and that we get a functor $\bW(k) : (\TP^R)^\op \to Set$.

\begin{lemma} \label{l:Rmapwelldef}
Given an $R$-map $f : S \to T$, the composite $f^*_w \circ w$ is contained in the image of the ghost map.
\end{lemma}

\begin{proof}
We use Dwork's Lemma, compare \cite[Lemma 1.4]{He15}. First assume that $A=\bZ[a_t]_{t \in T}$ and $\phi_p(a_t)=a_t^p$, and that $(a_t)$ is the ``canonical Witt vector'' in $\bW_T(A)$. It then suffices to check that $(f^*_w \circ w)(a_t)$ is in the image of $w$. Let $\langle x_t \rangle = w(a_t) \in A^T$ and $\langle y_s \rangle = f^*_w \langle x_t \rangle \in A^S$. This means we have to verify that
\[
 y_s \equiv \phi_p(y_{s/p}) \mod p^{\nu_p(|s|)}
\]
whenever $p \mid |s|$. We have $y_s=x_{f(s)}$ and $y_{s/p}=x_{f(s/p)}=x_{f(s)/p}$, so because $\langle x_t \rangle$ is in the image of the ghost map we can conclude that $y_s \equiv \phi_p(y_{s/p}) \mod p^{\nu_p(|f(s)|)}$. But then the result follows, because $|s| \mid |f(s)|$ and so $\nu_p(|f(s)|) \geq \nu_p(|s|)$.

Now the result follows for any $k$ and any element of $\bW_T(k)$ by naturality. Given $(a_t') \in \bW_T(k)$, let $g : A \to k$ be the ring homomorphism that sends $a_t$ to $a_t'$. Then the only way to define $f^*(a_t')$ in a way that is natural in $k$ is as $\bW_S(g)(f^*(a_s))$, and because the ghost map is natural in the ring this does indeed make the diagram in Definition \ref{d:RFdiagmap} commute.
\end{proof}

It follows that Definition \ref{d:RFdiagmap} does indeed define a map $f^* : \bW_T(k) \to \bW_S(k)$. It is unique because it is unique on the ``universal Witt vector'' $(a_t)$ in the ``representing ring'' $k=\bZ[a_t]_{t \in T}$. Next we discuss how $f^*$ generalizes the diagonal map, the classical restriction map, and the classical Frobenius map.

\begin{lemma}
Let $S$ be any truncation poset and let $\nabla : S \coprod S \to S$ be the fold map. Then
\[
 \nabla^* : \bW_S(k) \to \bW_{S \coprod S}(k) \cong \bW_S(k) \times \bW_S(k)
\]
is the diagonal map.
\end{lemma}

\begin{proof}
This is immediate from the definition.
\end{proof}

\begin{lemma}
Suppose $S \subset T$ is an inclusion of ordinary truncation sets and let $i : S \to T$ denote the inclusion. Then
\[
 i^* : \bW_T(k) \to \bW_S(k)
\]
is the classical restriction map $R^T_S$.
\end{lemma}

\begin{proof}
This is immediate from the definition, using that the restriction map can be defined on either Witt coordinates or ghost coordinates.
\end{proof}

\begin{lemma}
Let $n \in \bN$ and let $S$ be an ordinary truncation set. Then
\[
 f^* : \bW_S(k) \to \bW_{S/n}(k)
\]
induced by the multiplication by $n$ map $f : S/n \to S$ is the classical Frobenius map $F_n$.
\end{lemma}

\begin{proof}
In this case Definition \ref{d:RFdiagmap} reduces to the usual definition of the Frobenius.
\end{proof}

Given any $R$-map $f : S \to T$, the induced map $f^* : \bW_T(k) \to \bW_S(k)$ factors as a composite of an iterated diagonal map, a classical Frobenius map on each connected component, and a classical restriction map on each connected component. Because each of these maps is well defined on Witt coordinates by the classical theory of Witt vectors, it is possible to prove Lemma \ref{l:Rmapwelldef} by piecing together these classical results.

\subsection{Addition and Verschiebung maps}
The maps in $\TP^T$ will encode addition and Verschiebung maps. Again we make the definition first and justify it later.

\begin{definition} \label{d:Vplusmap}
Let $f : S \to T$ be a $T$-map of truncation posets. Then
\[
 f_\oplus : \bW_S(k) \to \bW_T(k)
\]
is defined to be the unique map of sets, natural in $k$, such that the diagram
\[ \xymatrix{
 \bW_S(k) \ar[r]^-w \ar[d]_{f_\oplus} & k^S \ar[d]^{f^w_\oplus} \\
 \bW_T(k) \ar[r]^-w & k^T
} \]
commutes. Here $f^w_\oplus$ is the map defined by
\[
 (f^w_\oplus \langle x_s \rangle)_t = \sum_{s \in f^{-1}(t)} \frac{|t|}{|s|} x_s.
\]
\end{definition}

Note that we needed the finiteness condition in Definition \ref{d:Tmap} to define $f_\oplus^w$.

\begin{lemma} \label{l:Vmapwelldef}
Given a $T$-map $f : S \to T$, the composite $f_\oplus^w \circ w$ is contained in the image of the ghost map.
\end{lemma}

\begin{proof}
First let $A=\bZ[a_s]_{s \in S}$ and let $(a_s) \in \bW_S(A)$ be the ``canonical Witt vector''. As in the proof of Lemma \ref{l:Rmapwelldef}, let $\langle x_s \rangle = w(a_s)$ and $\langle y_t \rangle = f_\oplus^w \langle x_s \rangle$. Also, let $\phi_p : A \to A$ be the ring map defined by mapping $a_s$ to $a_s^p$.

Suppose $\nu_p(|t|) \geq 1$. We need to verify that $y_t \equiv \phi_p(y_{t/p}) \mod p^{\nu_p(|t|)}$. We have
\[
 y_t = \sum_{f(s)=t} \frac{|t|}{|s|} x_s = \sum_{f(s)=t} \frac{|t|}{|s|} \sum_{u \mid s} |u| a_u^{s/u}
\]
and
\[
 y_{t/p} = \sum_{f(s')=t/p} \frac{|t|/p}{|s'|} x_{s'} = \sum_{f(s')=t/p} \frac{|t|/p}{|s'|} \sum_{v \mid s'} |v| a_v^{s'/v}.
\]
Hence
\[
 \phi_p(y_{t/p}) = \sum_{f(s')=t/p} \frac{|t|/p}{|s'|} \sum_{v \mid s'} |v| a_v^{ps'/v}.
\]

Each term in $\phi_p(y_{t/p})$ labelled by $s'$ and $v$ corresponds to a term in $y_t$ labelled by $s=ps'$ and $u=v$. Note that here we used the fibration condition in Definition \ref{d:Tmap}. The terms of $y_t$ that do not correspond to a term of $\phi_p(y_{t/p})$ all correspond to terms labelled by $(s,u)$ with $\nu_p(|u|) = \nu_p(s)$. But in those cases the coefficient $\frac{|t|}{|s|} \cdot |u|$ of $a_u^{s/u}$ is divisible by $p^{\nu_p(|t|)}$, so the result follows.

Now the result follows for any $k$ and any element of $\bW_S(k)$ by naturality, using the same argument as in the proof of Lemma \ref{l:Rmapwelldef}.
\end{proof}

It follows that Definition \ref{d:Vplusmap} does indeed define a unique map $f_\oplus : \bW_S(k) \to \bW_T(k)$.

\begin{lemma}
Let $S$ be an ordinary truncation set and let $\nabla : S \coprod S \to S$ be the fold map. Then
\[
 \bW_S(k) \times \bW_S(k) \cong \bW_{S \coprod S}(k) \xto{\nabla_\oplus} \bW_S(k)
\]
is the classical addition map on $\bW_S(k)$.
\end{lemma}

\begin{proof}
This follows immediately from Definition \ref{d:Vplusmap} because in this case each $\frac{|t|}{|s|} = 1$ and the classical addition map on Witt vectors is defined by using the addition on ghost coordinates.
\end{proof}

Of course the fold map also furnishes $\bW_S(k)$ with an addition map when $S$ is a truncation poset, in the same way. By another application of Dwork's Lemma we find that $-w(a_s)$ is in the image of the ghost map, and this implies that $\bW_S(k)$ has additive inverses. Hence $\bW_S(k)$ is indeed an abelian group.

\begin{lemma} \label{l:fstaradditive}
Let $f : S \to T$ be an $R$-map of truncation posets. Then $f^* : \bW_T(k) \to \bW_S(k)$ is a group homomorphism.
\end{lemma}

\begin{proof}
It suffices to prove that the diagram
\[ \xymatrix{
 \bW_T(k) \times \bW_T(k) \ar[r]^-{f^* \times f^*} \ar[d]_{\nabla_\oplus} & \bW_S(k) \times \bW_S(k) \ar[d]^{\nabla_\oplus} \\
 \bW_T(k) \ar[r]^-{f^*} & \bW_S(k)
} \]
commutes on ghost coordinates, which is clear because $(\langle x_t \rangle, \langle y_t \rangle)$ maps to $\langle x_{f(s)} + y_{f(s)} \rangle$  both ways.
\end{proof}

\begin{lemma} \label{l:foplusadditive}
Let $f : S \to T$ be a $T$-map of truncation posets. Then $f_\oplus : \bW_S(k) \to \bW_T(k)$ is a group homomorphism.
\end{lemma}

\begin{proof}
This is clear because the diagram
\[ \xymatrix{
 S \coprod S \ar[r]^-{f \coprod f} \ar[d]_{\nabla} & T \coprod T \ar[d]^{\nabla} \\
 S \ar[r]^-{f} & T
} \]
commutes in $\TP^T$.
\end{proof}

\begin{lemma}
Let $S$ be an ordinary truncation set, let $n \in \bN$, and let $f : S/n \to S$ be the multiplication by $n$ map. Then
\[
 \bW_{S/n}(k) \xto{f_\oplus} \bW_S(k)
\]
is the classical Verschiebung map $V_n$.
\end{lemma}

\begin{proof}
In this case $f$ is injective and each $\frac{|t|}{|s|} = n$, so Definition \ref{d:Vplusmap} says that on ghost coordinates we have
\[
 (f^w_\oplus \langle x_s \rangle)_t = \begin{cases} nx_{t/n} \quad & \textnormal{if $n \mid t$} \\ 0 \quad & \textnormal{if $n \nmid t$} \end{cases}
\]
But this is one equivalent definition of $V_n$.
\end{proof}

Any $T$-map factors as a composite of addition maps and Verschiebung maps, so it is possible to combine the existence of addition and Verschiebung maps on classical Witt vectors to prove Lemma \ref{l:Vmapwelldef}.

\subsection{Multiplication and Norm maps}
Finally, the maps in $\TP^N$ will encode multiplication and norm maps. Once again we start with the definition. Recall the definition of $\widehat{f^{-1}}(t)$ from Definition \ref{d:Nmap}.

\begin{definition} \label{d:Ntimesmap}
Let $f : S \to T$ be an $N$-map. Then
\[
 f_\otimes : \bW_S(k) \to \bW_T(k)
\]
is defined to be the unique map of sets, natural in $k$, such that the diagram
\[ \xymatrix{
 \bW_S(k) \ar[r]^-w \ar[d]_{f_\otimes} & k^S \ar[d]^{f^w_\otimes} \\
 \bW_T(k) \ar[r]^-w & k^T
} \]
commutes. Here $f^w_\otimes$ is the map defined by
\[
 (f^w_\otimes \langle x_s \rangle)_t = \prod_{s \in \widehat{f^{-1}}(t)} x_s^{|t|/|s|}.
\]
\end{definition}

Note that we needed the strong finiteness condition in Definition \ref{d:Nmap} to make sense of the map $f^w_\otimes$.

\begin{lemma} \label{l:Nmapwelldef}
Given an $N$-map $f : S \to T$, the composite $f^w_\otimes \circ w$ is contained in the image of the ghost map.
\end{lemma}

\begin{proof}
First let $A=\bZ[a_s]_{s \in S}$, $(a_s) \in \bW_S(A)$, $\langle x_s \rangle = w(a_s)$ and $\langle y_t \rangle = f_\otimes^w \langle x_s \rangle$ as in the proof of Lemma \ref{l:Vmapwelldef} above, and let $\phi_p$ be the ring map defined by $a_s \mapsto a_s^p$. To unclutter the notation we write $t/s$ for $|t|/|s|$.

Suppose $\nu_p(|t|) \geq 1$. We need to verify that $y_t \equiv \phi_p(y_{t/p}) \mod p^{\nu_p(|t|)}$. We have
\[
 y_t = \prod_{s \in \widehat{f^{-1}}(t)} x_s^{t/s} = \prod_{s \in \widehat{f^{-1}}(t)} \Big( \sum_{u \mid s} |u|a_u^{s/u} \Big)^{t/s}
\]
and
\[
 y_{t/p} = \prod_{s' \in \widehat{f^{-1}}(t/p)} x_{s'}^{t/ps} = \prod_{s' \in \widehat{f^{-1}}(t/p)} \Big( \sum_{v \mid s'} |v|a_v^{s'/v} \Big)^{t/ps'}.
\]
It follows that
\[
 \phi_p(y_{t/p}) = \prod_{s' \in \widehat{f^{-1}}(t/p)} \Big( \sum_{v \mid s'} |v|a_v^{ps'/v} \Big)^{t/ps'}.
\]

To proceed we need to understand the relationship between $\widehat{f^{-1}}(t)$ and $\widehat{f^{-1}}(t/p)$. Consider $s' \in \widehat{f^{-1}}(t/p)$. Then we get an $s \in \widehat{f^{-1}(t)}$ as in the following two cases.
\begin{enumerate}
 \item $p \nmid \frac{|f(s')|}{|t/p|}$. Then $s = ps' \in \widehat{f^{-1}}(t)$.
 \item $p \mid \frac{|f(s')|}{|t/p|}$. Then $s=s' \in \widehat{f^{-1}}(t)$.
\end{enumerate}
Note that for this we had to use the strong fibration condition in Definition \ref{d:Nmap}.

In each case it is straightforward to verify that the factor corresponding to $s'$ in $\phi_p(y_{t/p})$ and the factor corresponding to $s$ in $y_t$ are congruent $\mod p^{\nu_p(|t|)}$.

Once again the result follows for any $k$ by naturality.
\end{proof}

As for the other types of maps it follows that Definition \ref{d:Ntimesmap} defines a unique map $f_\otimes : \bW_S(k) \to \bW_T(k)$.

\begin{lemma}
Let $S$ be an ordinary truncation set and let $\nabla : S \coprod S \to S$ be the fold map. Then
\[
 \bW_S(k) \times \bW_S(k) \cong \bW_{S \coprod S}(k) \xto{\nabla_\otimes} \bW_S(k)
\]
is the classical multiplication map on $\bW_S(k)$.
\end{lemma}

\begin{proof}
This is clear because the classical multiplication map is defined via the multiplication map on ghost coordinates.
\end{proof}

Of course the fold map also furnishes $\bW_S(k)$ with a multiplication map when $S$ is a truncation poset, in the same way. And it is clear that multiplication distributes over addition because this holds on ghost coordinates, so our definitions make $\bW_S(k)$ into a commutative ring.

\begin{lemma}
Let $f : S \to T$ be an $R$-map of truncation posets. Then $f^* : \bW_T(k) \to \bW_S(k)$ is multiplicative.
\end{lemma}

\begin{proof}
This is similar to the proof of Lemma \ref{l:fstaradditive}.
\end{proof}

It follows that $f^* : \bW_T(k) \to \bW_S(k)$ is a ring homomorphism.

\begin{lemma}
Let $f : S \to T$ be an $N$-map of truncation posets. Then $f_\otimes : \bW_S(k) \to \bW_T(k)$ is multiplicative.
\end{lemma}

\begin{proof}
This is similar to the proof of Lemma \ref{l:foplusadditive}.
\end{proof}

\begin{lemma}
Let $S$ be an ordinary truncation set, let $n \in \bN$, and let $f : S \to \langle n \rangle S$ be the multiplication by $n$ map. Then
\[
 \bW_S(k) \xto{f_\otimes} \bW_{\langle n \rangle S}(k)
\]
is the ``classical'' norm map $N_n$.
\end{lemma}

\begin{proof}
A comparison of the map $f_\otimes^w$ with the formula for $N_n$ on ghost coordinates from \cite{An_norm} shows that they agree.
\end{proof}

\section{Combining $T$-maps and $R$-maps} \label{s:Mackey}
In this section we define a category $\TP^{TR}$ by combining the categories $\TP^T$ and $(\TP^R)^\op$. The following definition is more complicated than it needs to be; we present it this way in anticipation of the category $\TP^{TNR}$.

\begin{definition}
The category $\TP^{TR}$ has objects the truncation posets, and a morphism $S \to T$ in $\TP^{TR}$ is an equivalence class of diagrams
\[
 S \overset{f_1}{\longrightarrow} A_1 \overset{f_2}{\longrightarrow} A_2 \overset{f_3}{\longrightarrow} \ldots \overset{f_n}{\longrightarrow} A_n \overset{f_{n+1}}{\longrightarrow} T
\]
where each $f_i$ is a map in one of $\TP^T$ and $(\TP^R)^\op$. The equivalence relation on such diagrams is generated by the following types of relations:
\begin{enumerate}
 \item Isomorphism of diagrams.
 \item Insertion of an identity morphism.
 \item Composition if $f_i$ and $f_{i+1}$ are in the same category $\TP^T$ or $(\TP^R)^\op$.
 \item Commuting an $R$-map past a $T$-map as in Definition \ref{d:commuteRT} below.
\end{enumerate}
\end{definition}

Given an $R$-map $f : S \to T$ we abuse notation and write $f^* : T \to S$ for the corresponding map in $\TP^{TR}$, and given a $T$-map $f : S \to T$ we write $f_\oplus : S \to T$ for the corresponding map in $\TP^{TR}$. (Hence with our notation the functor $\bW(k)$ takes $f^*$ to $f^*$ and $f_\oplus$ to $f_\oplus$.) We need to explain how to commute an $R$-map past a $T$-map.

\begin{definition} \label{d:commuteRT}
Given a diagram
\[
 S \xto{f} A \xfrom{g} T
\]
with $f \in \TP^T$ and $g \in \TP^R$, we declare the composite $g^* \circ f_\oplus$ to be equal to the composite $f'_\oplus \circ (g')^*$, where
\[
 S \xfrom{g'} f_\oplus^* T \xto{f'} T
\]
and $f_\oplus^* T$ is defined by
\[
 f_\oplus^* T = \Big\{ (s,t, \xi) \quad | \quad f(s)=g(t) \textnormal{ and } \xi \in C_m, \,\, m=\gcd\Big(\frac{|f(s)|}{|s|}, \frac{|g(t)|}{|t|}\Big) \Big\}.
\]
Here $f' : f_\oplus^* T \to T$ and $g' : f_\oplus^* T \to S$ are the obvious maps, sending $(s,t,\xi)$ to $t$ and $s$, respectively. The norm on $f_\oplus^* T$ is defined by
\[
 |(s,t,\xi)| = \gcd(|s|,|t|),
\]
and we say $(s_1,t_1, \xi_1) \mid (s_2,t_2,\xi_2)$ if $s_1 \mid s_2$, $t_1 \mid t_2$, and $\xi_1=\xi_2$. Here we identify $C_{m_1}$ and $C_{m_2}$, using that $|f(s_2)|/|s_2|=|f(s_1)|/|s_1|$ and $|g(t_2)|/|t_2|=|g(t_1)|/|t_1|$.
\end{definition}

In other words, we take the usual pullback $S \times_A T = f^* T$ but count each $(s,t)$ with multiplicity to account for the fact that $|(s,t)|=\gcd(|s|,|t|)$ rather than the expected $|f(s)|=|g(t)|$. The cyclic group can be thought of as a bookkeeping device.

\begin{lemma}
With notation as above, $f_\oplus^* T$ is a truncation poset, $g' : f_\oplus^* T \to S$ is an $R$-map, and $f' : f_\oplus^* T \to T$ is a $T$-map.
\end{lemma}

\begin{proof}
This is straightforward. For example, given $(s,t,\xi) \in f_\oplus^* T$ and $t' \in T$ with $t \mid t'$ we need to find $(s',t',\xi') \in f_\oplus^* T$ with $(s,t,\xi) \mid (s',t',\xi')$. Because $f$ is a $T$-map, so in particular a fibration, there is some $s' \in S$ with $s \mid s'$ and $f(s')=g(t')$. Upon identifying $C_m$ with $C_{m'}$, where $m=\gcd\big( \frac{|f(s)|}{|s|}, \frac{|g(t)|}{|t|} \big)$ as before and $m'=gcd\big( \frac{|f(s')|}{|s'|}, \frac{|g(t')|}{|t'|} \big)$, we can take $\xi'=\xi$ and $(s',t',\xi')$ is the required element of $f_\oplus^* T$.
\end{proof}

Because we can always commute an $R$-map past a $T$-map it follows that the category $\TP^{TR}$ has a much simpler description:

\begin{proposition} \label{p:TPTRuniquespan}
Any map in the category $\TP^{TR}$ defined above can be written uniquely, up to isomorphism of spans, as a composite $f'_\oplus \circ (g')^*$ where $g'$ is an $R$-map and $f'$ is a $T$-map for a diagram
\[
 S \xfrom{g'} A' \xto{f'} T.
\]
\end{proposition}

\begin{proof}
Given a map from $S$ to $T$ as in Definition \ref{d:commuteRT}, we can commute all the $R$-maps past the $T$-maps and compose to obtain a description of the map as $(f')_\oplus \circ (g')^*$ for some $R$-map $g'$ and $T$-map $f'$.

It is possible to give a direct proof that the category $\TP^{TR}$ does not collapse further, but it also follows, after using Lemma \ref{l:commuteRTonWitt} below, because inequivalent spans $S \xfrom{g_1} A_1 \xto{f_1} T$ and $S \xfrom{g_2} A_2 \xto{f_2} T$ give different maps $\bW_S(k) \to \bW_T(k)$ for $k = \bZ$.

To see that inequivalent spans give different maps $\bW_S(\bZ) \to \bW_T(\bZ)$, it suffices to show that given a map $h : \bW_S(\bZ) \to \bW_T(\bZ)$ which on ghost coordinates is given by $h(\langle x_s \rangle)_t = \sum c_{s,t} x_s$ there is at most one span $S \xfrom{g} A \xto{f} T$ with $h = f_\oplus \circ g^*$. To do this, we observe that we must have $A = \coprod A_t$ where $A_t = \{a \in A \quad | \quad f(\frac{a}{|a|}) = t\}$, and that we can describe each $A_t$ explicitly. If $|t|=1$, we must have
\[
 A_t = \{da_{s,t,\xi} \quad | \quad c_{s,t} > 0, \, \xi \in C_{c_{s,t}}, \, d \in \bN \textnormal{ with $dt \in T$} \}.
\]
Here $d_1 a_{s_1,t,\xi_1} \mid d_2 a_{s_2,t,\xi_2}$ if $d_1 \mid d_2$, $s_1 = s_2$ and $\xi_1 = \xi_2$, and the maps are given by $f(da_{s,t,\xi}) = dt$ and $g(da_{s,t,\xi}) = ds$.

For $|t| > 1$, we can define $A_t$ similarly after taking into account the map $h_{<t}$ induced by the span $S \leftarrow \displaystyle\coprod_{t' | t, \, t' \neq t} A_{t'} \to T$.
\end{proof}

Definition \ref{d:commuteRT} is justified by the following result.

\begin{lemma} \label{l:commuteRTonWitt}
Witt notation as in Definition \ref{d:commuteRT} the composite
\[
 \bW_S(k) \xto{f_\oplus} \bW_A(k) \xto{g^*} \bW_T(k)
\]
is equal to the composite
\[
 \bW_S(k) \xto{(g')^*} \bW_{f_\oplus^* T}(k) \xto{f'_\oplus} \bW_T(k).
\]
\end{lemma}

\begin{proof}
It suffices to prove this on ghost coordinates. Take $\langle x_s \rangle \in k^S$, and suppose the first composite maps this to $\langle y_t \rangle$ and the second composite maps it to $\langle y_t' \rangle$. Then we get
\begin{eqnarray*}
 y_t & = & \sum_{s \in f^{-1}(g(t))} \frac{|g(t)|}{|s|} x_s \\
 & = & \sum_{s \in f^{-1}(g(t))} \frac{|t|}{\gcd(|s|,|t|)} \cdot \gcd\Big(\frac{|f(s)|}{|s|}, \frac{|g(t)|}{|t|}\Big) x_s \\
 & = & \sum_{(s,t,\xi) \in f_\oplus^* T} \frac{|t|}{\gcd(|s|,|t|)} x_s \\
 & = & y_t',
\end{eqnarray*}
which proves the result.
\end{proof}

We have now incorporated restriction maps, Frobenius maps, addition maps and Verschiebung maps in one category of truncation posets. Putting it all together we have proved the following:

\begin{theorem} \label{t:TRfunctor}
Let $k$ be a commtuative ring. There is a functor
\[
 \bW(k) : \TP^{TR} \to Set
\]
sending $S$ to $\bW_S(k)$, such that the composite $(\TP^R)^\op \to \TP^{TR} \to Set$ agrees with the functor in Definition \ref{d:RFdiagmap} and the composite $\TP^T \to \TP^{TR} \to Set$ agrees with the functor in Definition \ref{d:Vplusmap}.
\end{theorem}

\section{Combining $T$-maps, $N$-maps and $R$-maps} \label{s:Tambara}
Finally we define a category $\TP^{TNR}$ by combining $\TP^T$, $\TP^N$ and $(\TP^R)^\op$.

\begin{definition} \label{d:TrunTNR}
The category $\TP^{TNR}$ has objects the truncation posets, and a morphism $S \to T$ in $\TP^{TNR}$ is an equivalence class of diagrams
\[
 S \overset{f_1}{\longrightarrow} A_1 \overset{f_2}{\longrightarrow} A_2 \overset{f_3}{\longrightarrow} \ldots \overset{f_n}{\longrightarrow} A_n \overset{f_{n+1}}{\longrightarrow} T
\]
where each $f_i$ is a map in one of $\TP^T$, $\TP^N$ and $(\TP^R)^\op$. The equivalence relation on such diagrams is generated by the following:
\begin{enumerate}
 \item Isomorphism of diagrams.
 \item Insertion of an identity morphism.
 \item Composition if $f_i$ and $f_{i+1}$ are in the same category $\TP^T$, $\TP^N$ or $(\TP^R)^\op$.
 \item Commuting an $R$-map past a $T$-map as in Definition \ref{d:commuteRT} above.
 \item Commuting an $R$-map past an $N$-map as in Definition \ref{d:commuteNR} below if the pullback $f_\otimes^* T$ exists as in Definition \ref{d:multpullback2} below.
 \item Commuting an $N$-map past a $T$-map as in Definition \ref{d:commuteTN} below.
\end{enumerate}
\end{definition}

Given an $N$-map $f : S \to T$ we write $f_\otimes : S \to T$ for the corresponding map in $\TP^{TNR}$. Note that it is not always possible to commute an $R$-map past an $N$-map, as the following example shows.

\begin{example} \label{ex:notcommutingRpastN}
Let $f : \{1,3\} \to \{1,2,3,6\}$ be the multiplication by $2$ map and let $g : \{1,2,3\} \to \{1,2,3,6\}$ be the inclusion. Then the composite
\[
 \bW_{\{1,3\}}(k) \xto{f_\otimes} \bW_{\{1,2,3,6\}}(k) \xto{g^*} \bW_{\{1,2,3\}}(k)
\]
is given on ghost coordinates by
\[
 \langle x_1, x_3 \rangle \mapsto \langle x_1,x_1^2, x_3,x_3^2 \rangle \mapsto \langle x_1,x_1^2, x_3 \rangle.
\]
But it is impossible to define this map as a composite of an $R$-map followed by an $N$-map. With an $R$-map we can make as many copies as we want of $\langle x_1,x_3 \rangle$, $\langle x_1 \rangle$ and $\langle x_3 \rangle$, using the truncation set $\coprod_{i_1} \{1,3\} \coprod_{i_2} \{1\} \coprod_{i_3} \{3\}$. But this truncation set only maps to $\{1,2,3\}$ via an $N$-map if $i_1=i_2=i_3=0$. 
\end{example}

The analogue of the additive pullback $f^*_\oplus T$ considered in Definition \ref{d:commuteRT} above is the following:

\begin{definition} \label{d:multpullback1}
Given a diagram
\[
 S \xto{f} A \xfrom{g} T,
\]
with $f \in \TP^N$ and $g \in \TP^R$, let
\[
 f_\otimes^* T = \Big\{(s,t,\xi) \quad | \quad g(t) \mid f(s),\,\, s \textnormal{ minimal, } t \textnormal{ maximal, } \xi \in C_m \Big\}.
\]
Here $m = \gcd \big( \frac{|f(s)|}{|s|}, \frac{|g(t)|}{|t|} \big)$ as before, $s$ minimal means that there is no $s' \mid s$ with $s' \neq s$ and $g(t) \mid f(s')$ (for $t$ fixed), and $t$ maximal means that there is no $t \mid t'$ with $t \neq t'$ and $g(t') \mid f(s)$ (for $s$ fixed).
\end{definition}

This is not necessarily a good definition, because if we try to carry this out in the situation in Example \ref{ex:notcommutingRpastN} we find that the obvious maps $f' : f^*_\otimes T \to T$ and $g' : f^*_\otimes T \to S$ are not maps of truncation posets.

\begin{definition} \label{d:multpullback2}
We say the pullback $f^*_\otimes T$ from Definition \ref{d:multpullback1}  exists (as a truncation poset) if the following additional condition is satisfied. For any $(s_1,t_1,\xi_1)$ and $(s_2,t_2,\xi_2)$ in $f^*_\otimes T$ with $s_1$ and $s_2$ in the same connected component of $S$ and $t_1$ and $t_2$ in the same connected component of $T$ we have $|s_1||t_2|=|s_2||t_1|$. We then define
\[
 |(s,t,\xi)| = \gcd(|s|,|t|)
\]
and
\[
 (s_1,t_1,\xi_1) \mid (s_2, t_2,\xi_2) \qquad \textnormal{if} \qquad s_1 \mid s_2, t_1 \mid t_2 \textnormal{ and } \xi_1=\xi_2,
\]
where as usual we have identified $C_{m_1}$ and $C_{m_2}$.
\end{definition}

The pullback $f_\otimes^* T$ often exists. For example, the following gives a sufficient condition.

\begin{definition}
We say a truncation poset $T$ \emph{has joins} if $t \mid t_1$ and $t \mid t_2$ implies that there exists $t'$ with $t_1 \mid t'$ and $t_2 \mid t'$.
\end{definition}

For example, for any $n \in \bN$ the truncation set $\langle n \rangle$ has joins. The truncation set $\{1,2,3\}$ does not have joins.

\begin{lemma} \label{l:joinsimpliespullback}
Suppose $g : T \to A$ is an $R$-map and suppose $T$ has joins. Then the pullback $f_\otimes^* T$ exists for any $N$-map $f : S \to A$.
\end{lemma}

\begin{proof}
Suppose we have $(s_1,t_1,\xi_1)$ and $(s_2,t_2,\xi_2)$ in $f_\otimes^* T$ with $s_1$ and $s_2$ in the same connected component of $S$ and $t_1$ and $t_2$ in the same connected component of $T$. Then we need to show that $|s_1||t_2| = |s_2||t_1|$. Let $t'$ be the join of $t_1$ and $t_2$.

Because $f$ is an $N$-map, there are $s_i' \in S$ for $i=1,2$ with $s_i \mid s_i'$ and $g(t') \mid f(s_i')$, and with $s_i'$ minimal. But because $s_1'$ and $s_2'$ are in the same connected component of $S$ and satisfy the same minimality condition we must have $s_1'=s_2'$. It follows that
\[
 \frac{|f(s_1)|}{|g(t_1)|} = \frac{|f(s_1')|}{|g(t')|} = \frac{|f(s_2')|}{|g(t')|} = \frac{|f(s_2)|}{|g(t_2)|} 
\]
and we are done.
\end{proof}

\begin{definition} \label{d:commuteNR}
Let
\[
 S \xto{f} A \xfrom{g} T
\]
be a diagram with $f \in \TP^N$ and $g \in \TP^R$, and suppose the pullback $f_\otimes^* T$ exists. Then we declare the composite $g^* \circ f_\otimes$ to be equal to the composite $f'_\otimes \circ (g')^*$.
\end{definition}

We justify the above definition with the following two results.

\begin{lemma} \label{l:pullbackofNmapisnmap}
Let $f : S \to A$ be an $N$-map, let $g : T \to A$ be an $R$-map, and suppose the pullback $f_\otimes^* T$ exists. Then $f' : f_\otimes^* T \to T$ is an $N$-map.
\end{lemma}

\begin{proof}
Suppose $(s,t,\xi) \in f^*_\otimes T$ and that $\bar{t} \in T$ is in the same connected component as $t$. We need to find $(s',t',\xi') \in f^*_\otimes T$ in the same connected component as $(s,t,\xi)$ with $\bar{t} \mid t'$.

Because $f$ is an $N$-map, we get an $s' \in S$ with $g(\bar{t}) \mid f(s')$. We can assume that $s'$ is minimal. Then we define $t'$ to be maximal with the property that $\bar{t} \mid t'$ and $g(t') \mid f(s')$. Then (after identifying $C_m$ and $C_{m'}$), $(s',t',\xi) \in f^*_\otimes T$ is the desired element.

Verifying the finiteness condition is straightforward and we omit it.
\end{proof}

\begin{lemma}
With assumptions as in Lemma \ref{l:pullbackofNmapisnmap}, the composite
\[
 \bW_S(k) \xto{f_\otimes} \bW_A(k) \xto{g^*} \bW_T(k)
\]
is equal to the composite
\[
 \bW_S(k) \xto{(g')^*} \bW_{f_\otimes^* T}(k) \xto{f'_\otimes} \bW_T(k).
\]
\end{lemma}

\begin{proof}
It suffices to show that the two maps agree on ghost coordinates. Suppose the first composite sends $\langle x_s \rangle$ to $\langle y_t \rangle$ and the second composite sends it to $\langle y_t' \rangle$. We find that
\[
 y_t = \prod_{s \in \widehat{f^{-1}}(g(t))} x_s^{|g(t)|/|s|},
\]
while
\[
 y_t' = \prod_{(s,t',\xi) \in \widehat{(f')^{-1}}(t)} x_s^{|t|/|(s,t',\xi)|}
\]
The point is that if $(s,t',\xi)$ is in $\widehat{(f')^{-1}}(t)$ then $s \in \widehat{f^{-1}}(g(t))$, and conversely, if $s \in \widehat{f^{-1}}(g(t))$ then there is a unique $t' \in T$ with $(s,t', \xi) \in \widehat{(f')^{-1}}(t)$ for any $\xi \in C_m$. Hence
\[
 y_t' = \prod_{s \in \widehat{f^{-1}}(g(t))} \Big[ x_s^{|t|/\gcd(|s|,|t'|)} \Big]^{\gcd(|f(s)|/|s|,|g(t)|/|t|)}.
\]
Hence it suffices to show that
\[
 st\gcd\Big(\frac{|f(s)|}{|s|}, \frac{|g(t)|}{|t|}\Big) = |g(t)| \gcd(|s|,|t'|).
\]
But this is easily verified using that $\gcd(|f(s)|/|g(t)|, |s|)=1$ and $\gcd(|s|,|t'|)=\gcd(|s|,|t|)$.
\end{proof}

This finishes our discussion of the composite of an $N$-map followed by an $R$-map. There is one more thing to do, namely to describe the composite of a $T$-map followed by an $N$-map. If we think of a $T$-map as addition and an $N$-map as multiplication this can be thought of as a distributivity law. This is somewhat combinatorial.

\begin{definition} \label{d:commuteTN}
Suppose $f : S \to A$ is a $T$-map and $g : A \to T$ is an $N$-map. Then we declare the composite $g_\otimes \circ f_\oplus$ to be equal to the composite $t_\oplus \circ n_\otimes \circ r^*$, where the maps $t$, $n$ and $r$ are defined by the exponential diagram
\[ \xymatrix{
 S \ar[d]_f & E \ar[l]_-r \ar[r]^-n & D \ar[d]^\tau \\
 A \ar[rr]^-g & & T
} \]
\end{definition}

It remains to say what we mean by an exponential diagram. Our definition is dictated by the proof of Lemma \ref{l:Wittofexp} below. We start with some auxiliary definitions. First, let
\[
 \wh{A} = \coprod_{t \in T} \coprod_{a \in \wh{g^{-1}}(t)} \{(t, a, \xi) \quad | \quad  \xi \in C_t/C_a \}
\]
and let $\wh{A}_t \subset \wh{A}$ be the subset whose first coordinate is $t$. Similarly, let
\[
 \wh{S} = \coprod_{t \in T} \coprod_{f(s) \in \wh{g}^{-1}(t) } \{(t, s, \zeta) \quad | \quad \zeta \in C_t/C_s \}
\]
and let $\wh{S}_t$ be the subset whose first coordinate is $t$. Let $\pi : \wh{S} \to S$ be the projection onto $S$.

We say $(t', a', \xi') \mid (t, a, \xi)$ in $\wh{A}$ if the following conditions hold. First, $t' \mid t$. Second, $a' \mid a$. And third, $\xi \mapsto \xi'$ under the quotient map $C_t/C_a \to C_{t'}/C_{a'}$. (Recall from Lemma \ref{l:minimaldiv} that for each $a \in \wh{g^{-1}}(t)$ there is exactly one $a' \in \wh{g^{-1}(t')}$ with $a' \mid a$, and that $\frac{|t'|}{|a'|}$ divides $\frac{|t|}{|a|}$; this last condition gives the quotient map $C_t/C_a \to C_{t'}/C_{a'}$.)

Similarly, we say that $(t', s', \zeta') \mid (t, s, \zeta)$ in $\wh{S}$ if the following conditions hold. First, $t' \mid t$. Second, $s' \mid s$. And third, $\zeta \mapsto \zeta'$ under the quotent map $C_t/C_s \to C_{t'}/C_{s'}$.

We have a map
\[
 \hat{f} : \wh{S} \to \wh{A} \qquad (t, s, \zeta) \mapsto (t, f(s), \bar{\zeta}),
\]
where $\bar{\zeta}$ is the image of $\zeta \in C_t/C_s$ under the quotient map $C_t/C_s \to C_t/C_{f(s)}$. Let
\[
 D' = \{(t, \sigma) \quad | \quad t \in T, \sigma : \wh{A}_t \to \wh{S}_t \textnormal{ section of $\hat{f}|_{\wh{S}_t}$} \}
\]
and let $D'_t$ be the subset whose first coordinate is $t$. 

We say that $(t', \sigma') \mid (t, \sigma)$ in $D'$ if the following conditions hold. First, $t' \mid t$. And second, if $(t', a', \xi') \mid (t, a, \xi)$ in $\wh{A}_t$ then $\sigma'(t', a', \xi') \mid \sigma(t, a, \xi)$ in $\wh{S}_t$.

Note that on ghost coordinates the composite $g_\otimes \circ f_\oplus$ is given on the $t$'th coordinate by a sum indexed over $D'_t$ as in the proof of Lemma \ref{l:Wittofexp} below. But this sum is too big, because the map $t_\oplus : \bW_D(k) \to \bW_T(k)$ also has a coefficient of $\frac{|t|}{|(t, \sigma)|}$. We fix this by letting $C_t$ act on $D'_t$ by conjugating the section and defining
\[
 D = \coprod_{t \in T} D_t'/C_t.
\]
We say that $(t', [\sigma']) \mid (t, [\sigma])$ in $D$ if $(t', \sigma') \mid (t, \sigma)$ in $D'$ for some choice of representatives, and we say that $(t, [\sigma]) \in D$ is divisible by $e$ if a representative $(t, \sigma) \in D'$ is divisible by $e$ in the sense that there is some $(t', \sigma')$ in $D'$ with $(t', \sigma') \mid (t, \sigma)$ and $\frac{|t|}{|t'|} = e$. We then define $|(t, [\sigma])|$ to be the largest $e$ such that $(t, [\sigma])$ is divisible by $e$.

\begin{lemma}
With the above definitions, $D$ is a truncation poset and the map $\tau : D \to T$ given by projection onto the first coordinate is a $T$-map.
\end{lemma}

\begin{proof}
The hardest part is verifying that $\tau : D \to T$ is a $T$-map. Given $(t, [\sigma]) \in D$ and $t' \in T$ with $t \mid t'$ we need to produce a section $\sigma' : \wh{A}_{t'} \to \wh{S}_{t'}$ with $(t, [\sigma]) \mid (t', [\sigma'])$. Pick a representative $\sigma$ for $[\sigma]$. For $a \in \wh{g^{-1}}(t)$, let $a' \in \wh{g^{-1}}(t')$ be the corresponding element as in Lemma \ref{l:minimaldiv}. For each $\xi' \in C_{t'}/C_{a'}$, let $\xi \in C_t/C_a$ be the image of $\xi'$ and let $s = \pi(\sigma(a,\xi))$. Since $f$ is a $T$-map, we get an $s' \in f^{-1}(a')$ with $s \mid s'$. We can then define $\sigma'(a', \xi') = (s', \zeta')$ where $\zeta' \in C_{t'}/C_{s'}$ is the pullback of $\xi'$ and $\zeta$ in the diagram
\[ \xymatrix{
 C_{t'}/C_{s'} \ar[r] \ar[d] & C_t/C_s \ar[d] \\
 C_{t'}/C_{a'} \ar[r] & C_t/C_a
} \]
(To verify that this is a pullback diagram, use that $\gcd(\frac{|t'|/|a'|}{|t|/|a|}, \frac{|a|}{|s|}) = 1$.)
\end{proof}

The definition of $E$ is similar. We start with
\[
 E' = \{(t, \sigma, a, \xi) \quad | \quad (t, \sigma) \in D', a \in \wh{g^{-1}}(t), \xi \in C_t/C_a \}. 
\]
As usual let $E'_t$ be the subset whose first coordinate is $t$. Let $C_t$ act on $E'_t$ by $\xi_t \cdot (t, \sigma, a, \xi) = (t, {}^{\xi_t} \sigma, a, \xi_t \cdot \xi)$ and define $E = \coprod_{t \in T} E'_t/C_t$. Then the map $r' : E' \to S$ defined by $r'(t, \sigma, a, \xi) = \pi(\sigma(a,\xi))$ induces a well defined map $r : E \to S$, and of course we have a map $n : E \to D$ that forgets $(a,\xi)$.

\begin{lemma}
With the above definitions, $E$ is a truncation poset, the map $r : E \to S$ is an $R$-map, and the map $n : E \to D$ is an $N$-map.
\end{lemma}

We omit the proof as it is a tedious but straightforward verification.

\begin{lemma} \label{l:Wittofexp}
Suppose we are given an exponential diagram as above. Then the composite
\[
 \bW_S(k) \xto{f_\oplus} \bW_A(k) \xto{g_\otimes} \bW_T(k)
\]
is equal to the composite
\[
 \bW_S(k) \xto{r^*} \bW_E(k) \xto{n_\otimes} \bW_D(k) \xto{t_\oplus} \bW_T(k).
\]
\end{lemma}

\begin{proof}
We can compute using ghost coordinates. Suppose the first composite sends $\langle x_s \rangle$ to $\langle y_t \rangle$ and the second composite sends $\langle x_s \rangle$ to $\langle y_t' \rangle$. Also let $\langle z_e \rangle = r^*_w \langle x_s \rangle$. We compute
\begin{eqnarray*}
 y_t & = & \prod_{a \in \wh{g^{-1}}(t)} \Big( \sum_{s \in f^{-1}(a)} \frac{|a|}{|s|} x_s \Big)^{|t|/|s|} \\
  & = & \prod_{(a, \xi) \in \wh{A}_t} \sum_{(s,\zeta) \in \hat{f}^{-1}(a,\xi)} x_{\pi(s,\zeta)} \\
  & = & \sum_{\sigma : \wh{A}_t \to \wh{S}_t} \prod_{(a,\xi) \in \wh{A}_t} z_{(t, \sigma, a, \xi)} \\
  & = & \sum_{[\sigma]} \frac{|t|}{|(t,[\sigma])|} \prod_{(a,\xi) \in \wh{A}_t} z_{(t, \sigma, a, \xi)},
\end{eqnarray*}
and this last expression is equal to $y_t'$ by inspection.
\end{proof}

Putting all of this together we have proved the following, which is a restatement of Theorem \ref{t:main}:

\begin{theorem} \label{t:mainbody}
Let $k$ be a commtuative ring. There is a functor
\[
 \bW(k) : \TP^{TNR} \to Set
\]
sending $S$ to $\bW_S(k)$, such that the composite $\TP^{TR} \to \TP^{TNR} \to Set$ agrees with the functor in Theorem \ref{t:TRfunctor} and the composite $\TP^N \to \TP^{TNR} \to Set$ agrees with the functor in Definition \ref{d:Ntimesmap}.
\end{theorem}

\section{The subcategory $\TP^{TNR}_\join$ and bispans} \label{s:bispans}
Motivated by Lemma \ref{l:joinsimpliespullback} we define a subcategory of $\TP^{TNR}$ as follows.

\begin{definition}
Let $\TP^{TNR}_\join$ be the category whose objects are truncation posets with join, and whose morphisms are generated by equivalence classes of morphisms of truncation posets with join in the same way as in Definition \ref{d:TrunTNR}.
\end{definition}

To make sense of this we should verify that the composite of two morphisms in $\TP^{TNR}_\join$ is still in $\TP^{TNR}_\join$. In other words, we should check that starting with truncation posets with joint, the truncation posets $f_\oplus^* T$ and $f_\otimes^* T$, as well as the truncation posets in the definition of an exponential diagram, all have join as well. This is straightforward and we omit it.

\begin{proposition}
Any map in $\TP^{TNR}_\join$ can be written uniquely, up to isomorphism of bispans, as a composite $h_\oplus \circ g_\otimes \circ f^*$ for a diagram
\[
 S \xfrom{f} A \xto{g} B \xto{h} T,
\]
where $A$ and $B$ are in $\TP^{TNR}_\join$ and $f \in \TP^R$, $g \in \TP^N$ and $h \in \TP^T$.
\end{proposition}

\begin{proof}
The only thing left to prove is that the category does not collapse further. This follows as in the proof of Proposition \ref{p:TPTRuniquespan} by noting that inequivalent bispans give different maps on Witt vectors of $\bZ$.
\end{proof}

Finally, we compare our construction to ``classical'' Tambara functors for cyclic groups.

\begin{definition} \label{d:TNF}
Let $\TP^{TNF}_{\langle n \rangle}$ be the subcategory of $\TP^{TNR}_\join$ whose objects are finite disjoint unions of $\langle m \rangle$ for $m \mid n$ and whose morphisms are given by equivalence classes of bispans
\[
 S \xfrom{f} A \xto{g} B \xto{h} T
\]
of such, with the following additional requirement: Suppose we decompose $A$ and $S$ as $\coprod A_i$ and $\coprod S_j$, respectively, and write $f$ as a coproduct of maps $A_i \xto{f_i} S_j \subset S$. If $A_i = \langle m_i \rangle$ and $S_j = \langle m_j \rangle$ then $m_i \mid m_j$ and $f_i$ is multiplication by $m_j/m_i$.
\end{definition}

Note that the maps $g : A \to B$ and $h : B \to T$ automatically satisfy the condition in Definition \ref{d:TNF}.

We obviously have a functor
\[
 \bW(k) : \TP^{TNF}_{\langle n \rangle} \to Set
\]
obtained by restricting the functor $\bW(k)$ from Theorem \ref{t:mainbody}. Now we can compare this to Tambara functors because of the following.

\begin{proposition}
There is an equivalence of categories between $\TP^{TNF}_{\langle n \rangle}$ and the category of bispans of finite $C_n$-sets.
\end{proposition}

\begin{proof}
The equivalence is given on objects by sending the truncation poset $\langle m \rangle$ to the finite $C_n$-set $C_n/C_m$ and sending disjoint unions to disjoint unions. We send a map $\frac{m_2}{m_1} : \langle m_1 \rangle \to \langle m_2 \rangle$ to the quotient map $C_n/C_{m_1} \to C_n/C_{m_2}$.

To finish the proof we should verify that the composition laws for the two types of bispans agree. This is straightforward and we omit it.
\end{proof}

\end{document}